\documentclass[10pt]{article}

\usepackage{amssymb}
\usepackage{eucal}
\usepackage{amsmath}
\usepackage{amsthm}
\usepackage{amscd}
\usepackage{graphics}
\usepackage{graphicx}
\usepackage{amsfonts}
\usepackage{latexsym}
\usepackage[all]{xy}

\newcommand{\R}{{\mathbb  R}}
\numberwithin{equation}{section}
\newtheorem{thm}{\bf Theorem}[section]

\newtheorem{prop}[thm]{\bf Proposition}
\newtheorem{defn}{\bf Definition}[section]

\theoremstyle{remark}

\begin{document}

\title{The stability problem and special solutions for the 5-components Maxwell-Bloch equations}
\author{Petre Birtea, Ioan Ca\c{s}u\footnote{Tel/fax: +40740928382/+40256592316; E-mail: casu@math.uvt.ro (Corresponding author)}\\
{\small Departamentul de Matematic\u a, Universitatea de Vest din Timi\c soara}\\
{\small Bd. V. P\^arvan, Nr. 4, 300223 Timi\c soara, Rom\^ania}\\
{\small E-mail: birtea@math.uvt.ro; casu@math.uvt.ro}}
\date{}

\maketitle



\begin{abstract}
\noindent For the 5-components Maxwell-Bloch system the stability problem for the isolated equilibria is completely solved. Using the geometry of the symplectic leaves, a detailed construction of the homoclinic orbits is given. Studying the problem of invariant sets for the system, we discover a rich family of periodic solutions in explicit form.
\end{abstract}

\section{Introduction}

After averaging and neglecting non-resonant terms, the unperturbed Maxwell-Bloch dynamics in the rotating wave approximation (RWA) is given by
\begin{equation*}
\left\{\begin{array}{l}
\dot X=Y\\
\dot Y=XZ\\
\dot Z=-\frac{1}{2}(XY^*+X^*Y),
\end{array}\right.
\end{equation*}
where $X,Y$ are complex scalar functions, that are denoting the self-consistent electric field and respectively the polarizability of the laser-matter, $Z$ is a real scalar function, which denotes the difference of its occupation numbers. The superscript $^*$ stands for the complex conjugate. For more details about the history and physical interpretations of this system see \cite{holm}, \cite{huang}, \cite{nath}.

Writing $X=x_1+\imath x_2,Y=y_1+\imath y_2$ and $Z=z$ the above system transforms into the 5-components Maxwell-Bloch system
\begin{equation}
\label{sistem5}
\left\{\begin{array}{l}
\dot x_1=y_1\\
\dot y_1=x_1z\\
\dot x_2=y_2\\
\dot y_2=x_2z\\
\dot z=-(x_1y_1+x_2y_2).
\end{array}\right.
\end{equation}
The Maxwell-Bloch system in the form \eqref{sistem5} has the advantage of a rich underlying geometrical structure that can be used in the study of its dynamical behavior. 

The system \eqref{sistem5} admits a Hamilton-Poisson formulation, where the Poisson tensor is given by
\begin{equation*}
%\label{Poisson}
J(x_1,y_1,x_2,y_2,z)=\left[\begin{array}{ccccc}
0&1&0&0&0\\
-1&0&0&0&x_1\\
0&0&0&1&0\\
0&0&-1&0&x_2\\
0&-x_1&0&-x_2&0
\end{array}\right]
\end{equation*}
and the Hamiltonian function is given by
\begin{equation*}
H(x_1,y_1,x_2,y_2,z)=\frac{1}{2}(y_1^2+y_2^2+z^2).
\end{equation*}
The system has two additional constants of motion, namely the Casimir of the Poisson structure $J$, given by
$$C(x_1,y_1,x_2,y_2,z)=\frac{1}{2}(x_1^2+x_2^2)+z$$
and a constant of motion derived from a bi-Hamiltonian structure of the system \eqref{sistem5} (see \cite{huang}) given by
$$I(x_1,y_1,x_2,y_2,z)=x_2y_1-x_1y_2.$$
A commuting property of the constants of motion $H$ and $I$ holds, i.e. 
$\{H,I\}=0$, where $\{\cdot,\cdot\}$ is the Poisson bracket associated to the Poisson tensor $J$.

\section{Stability of equilibria}

By a direct computation we obtain three families of equilibria for the system \eqref{sistem5}:
$${\cal E}_1=\{(0,0,0,0,M)|~M\in\R^*\};~~~{\cal E}_2=\{(M,0,N,0,0)|~M,N\in\R,M^2+N^2\neq 0\};~~~{\cal E}_3=\{(0,0,0,0,0)\}.$$

It is a well known fact that the dynamics of a Hamilton-Poisson system is foliated by the symplectic leaves associated to the Poisson structure. In our case the regular symplectic leaves are given by the connected components corresponding to pre-images of regular values of the Casimir function $C$. We denote by ${\cal O}_c=C^{-1}(c),c\in\R$ the regular symplectic leaves of the Poisson structure $J$.

The restriction of the dynamics \eqref{sistem5} to a regular leaf ${\cal O}_c$ becomes a completely integrable Hamiltonian system 
\begin{equation}
\label{redus}
({\cal O}_c,\omega_{{\cal O}_c},H|_{{\cal O}_c}),
\end{equation}
where the second commuting constant of motion is $I|_{{\cal O}_c}$. We will study the stability problem of equilibria on regular leaves ${\cal O}_c$ analogously to the approach used in \cite{nonlinear}.

The equilibria of the Hamiltonian system \eqref{redus} can be divided in two types:
\begin{align*}
\mathcal{K}_0&: = \left\{(x_1,y_1,x_2,y_2,z)\in {\cal O}_c \mid \mathbf{d}\left(H|_{{\cal O}_c}\right)(x_1,y_1,x_2,y_2,z) = 0, \;
\mathbf{d}\left(I|_{{\cal O}_c}\right)(x_1,y_1,x_2,y_2,z)=0\right\};\\
\mathcal{K}_1&: = \left\{(x_1,y_1,x_2,y_2,z)\in {\cal O}_c \mid \mathbf{d}\left(H|_{{\cal O}_c}\right)(x_1,y_1,x_2,y_2,z) = 0, \;
\mathbf{d}\left(I|_{{\cal O}_c}\right)(x_1,y_1,x_2,y_2,z)\neq 0\right\}.
\end{align*}

\begin{prop}
On a regular symplectic leaf ${\cal O}_c$ we have the following characterization for the equilibria:
$$\mathcal{K}_0={\cal O}_c\cap ({\cal E}_1\cup {\cal E}_3);~~~\mathcal{K}_1={\cal O}_c\cap {\cal E}_2.$$
\end{prop}

\begin{proof}
Because \eqref{redus} is a Hamiltonian system on a symplectic manifold the condition $\mathbf{d}\left(H|_{{\cal O}_c}\right)(e)=0$ is verified for any equilibrium point $e\in {\cal O}_c$.\\ 
Let $e_2\in {\cal O}_c\cap {\cal E}_2$. Then
\begin{equation*}
T_{e_2}{\cal O}_c=\{\bar v=(v_1,v_2,v_3,v_4,v_5)\in\R^5|~<\bar v,\nabla C(e_2)>=0\}=\{\bar v\in\R^5|~v_1M+v_3N+v_5=0\}.
\end{equation*}
We also have
$\mathbf{d}I(e_2)=N\mathbf{d}y_1-M\mathbf{d}y_2$.
Taking, for example, $\bar v=(-N,N,M,-M,0)\in T_{e_2}{\cal O}_c$ we have 
$\mathbf{d}\left(I|_{{\cal O}_c}\right)(e_2)(\bar v)=M^2+N^2\neq 0$,
which proves that $\mathbf{d}\left(I|_{{\cal O}_c}\right)(e_2)\neq 0$.

For the equilibria $e$ in ${\cal E}_1\cup {\cal E}_3$ the condition $\mathbf{d}\left(I|_{{\cal O}_c}\right)(e)= 0$ is trivially verified.
\end{proof}

The commutativity of the constants of motion $H|_{{\cal O}_c}$ and $I|_{{\cal O}_c}$ with respect to the symplectic form $\omega_{{\cal O}_c}$ implies that at an equilibrium point $e\in {\cal O}_c$ we have
$$\left[\mathbf{D}X_{H|_{{\cal O}_c}}(e), \mathbf{D}X_{I|_{{\cal O}_c}}(e)\right]=0,$$
where $\mathbf{D}X_{H|_{{\cal O}_c}}(e)$ and $\mathbf{D}X_{I|_{{\cal O}_c}}(e)$ are the derivatives of the vector fields  $X_{H|_{{\cal O}_c}}$ and $X_{I|_{{\cal O}_c}}$ at the equilibrium $e$ and consequently $\mathbf{D}X_{H|_{{\cal O}_c}}(e)$, $\mathbf{D}X_{I|_{{\cal O}_c}}(e)$ are infinitesimally symplectic relative to the symplectic form $\omega_{{\cal O}_c}(e)$ on the vector space  $T_e {\cal O}_c$.

\begin{defn}
\label{non-degenerate}
An equilibrium point $e\in \mathcal{K}_0$ is called non-degenerate if $\mathbf{D}X_{H|_{{\cal O}_c}}(e)$ and $\mathbf{D}X_{I|_{{\cal O}_c}}(e)$ generate a Cartan subalgebra of the Lie algebra of infinitesimal linear transformations of the symplectic vector space $\left(T_e {\cal O}_c, \omega_{{\cal O}_c}(e)\right)$. A Cartan subalgebra of the Lie algebra $\hbox{sp}(4,\R)$ is a two dimensional commutative sub-algebra which contains an element whose eigenvalues are all distinct.
\end{defn}

It follows that for a non-degenerate equilibrium belonging to $\mathcal{K}_0$ the matrices $\mathbf{D}X_{H|_{{\cal O}_c}}(e)$ and $\mathbf{D}X_{I|_{{\cal O}_c}}(e)$ can be simultaneously conjugated to one of the following four Cartan sub-algebras
\begin{equation}
\left.\begin{array}{cc}
\text{Type 1:} \quad 
\begin{bmatrix}
0&0&\!\!\!\!-A&0\\
0&0&0&\!\!\!\!-B\\
A&0&0&0\\
0&B&0&0
\end{bmatrix}&
\qquad\qquad \text{Type 2:} \quad
\begin{bmatrix}
-A&0&0&0\\
\;\;\;0&0&0&\!\!\!\!-B\\
\;\;\;0&0&A&0\\
\;\;\;0&B&0&0
\end{bmatrix} \\ 
~&~\\
\label{cartan_subalgebras}
\text{Type 3:} \quad 
\begin{bmatrix}
-A&0&0&0\\
\;\;\;0&\!\!\!\!-B&0&0\\
\;\;\;0&0&A&0\\
\;\;\;0&B&0&B
\end{bmatrix}&
\qquad\qquad \text{Type 4:} \quad
\begin{bmatrix}
-A&-B&0&\;\;0\\
\;\;\;B&-A&0&\;\;0\\
\;\;\;0&\;\;\;0&A&\!\!-B\\
\;\;\;0&\;\;\;0&B&\;\;A
\end{bmatrix} \end{array}\right.
\end{equation}
where $A,B \in \mathbb{R}$ (see, e.g., \cite{BoFo04}, Theorems 1.3 and 1.4).\\
Equilibria of type 1 are called {\it center-center} with the corresponding eigenvalues for the linearized system: $i A,-\imath A,i B,-\imath B$.\\
Equilibria of type 2 are called {\it center-saddle} with the corresponding eigenvalues for the linearized system: $A,-A,\imath B,-\imath B$.\\
Equilibria of type 3 are called {\it saddle-saddle} with the corresponding eigenvalues for the linearized system: $A,-A,B,-B$.\\
Equilibria of type 4 are called {\it focus-focus} with the corresponding eigenvalues for the linearized system: $A+\imath B,A-\imath B,-A+\imath B,-A-\imath B$.

\begin{thm}
We have the following stability behavior for the equilibria in ${\cal O}_c\cap {\cal E}_1$:
\begin{itemize}
\item[(i)] The equilibrium point ${\cal O}_c\cap {\cal E}_1=\{(0,0,0,0,c)\}$ for $c>0$ is a non-degenerate equilibrium of type focus-focus and consequently unstable.
\item[(ii)] The equilibrium point ${\cal O}_c\cap {\cal E}_1=\{(0,0,0,0,c)\}$ for $c<0$ is a non-degenerate equilibrium of type center-center and consequently stable.
\end{itemize}
\end{thm}

\begin{proof} (i) For the linearized systems at the equilibrium $(0,0,0,0,c)$ we have:
$$\mathbf{D}X_{H|_{{\cal O}_c}}(0,0,0,0,c)=
\left[\begin{array}{cccc}
0&1&0&0\\
c&0&0&0\\
0&0&0&1\\
0&0&c&0
\end{array}\right]$$
and its characteristic polynomial has the non-distinct eigenvalues $\sqrt{c},\sqrt{c},-\sqrt{c},-\sqrt{c}$ and respectively
$$\mathbf{D}X_{I|_{{\cal O}_c}}(0,0,0,0,c)=
\left[\begin{array}{cccc}
0&0&1&0\\
0&0&0&1\\
-1&0&0&0\\
0&-1&0&0
\end{array}\right]$$
and its characteristic polynomial has the non-distinct eigenvalues $\imath,\imath,-\imath,-\imath$.\\
To decide the type of stability we need to determine the non-degeneracy of the equilibrium $(0,0,0,0,c)$, i.e. we have to find a linear combination
$\mathbf{D}X_{H|_{{\cal O}_c}}(0,0,0,0,c)+\alpha \mathbf{D}X_{I|_{{\cal O}_c}}(0,0,0,0,c)$, where $\alpha$ is a non-zero real number, that has distinct eigenvalues. The characteristic polynomial of such a linear combination is given by
$$t^4+(2\alpha^2 -2c)t^2+(\alpha^2+c)^2.$$
After the substitution $t^2=s$ we obtain the quadratic polynomial
$$s^2+(2\alpha^2 -2c)s+(\alpha^2+c)^2,$$
which has the discriminant $\Delta = -16c\alpha^2 <0$ and therefore has two distinct complex roots. It follows that the characteristic polynomial  
$t^4+(2\alpha^2 -2c)t^2+(\alpha^2+c)^2$ has four distinct complex eigenvalues of the form $A+\imath B,A-\imath B,-A+\imath B,-A-\imath B$ with $A,B\in \R^*$. Consequently, the equilibrium $(0,0,0,0,c)$ for $c>0$ is a non-degenerate equilibrium of focus-focus type for the dynamics \eqref{redus} and therefore unstable for this dynamics. 

Similar computations lead to the proof of (ii).
\end{proof}

Although the equilibrium $(0,0,0,0,0)\in {\cal O}_0$ belongs to $\mathcal{K}_0$, it is a degenerate equilibrium in the sense of Definition \ref{non-degenerate}. Indeed, any linear combination $\alpha \mathbf{D}X_{H|_{{\cal O}_0}}(0,0,0,0,0)+\beta \mathbf{D}X_{I|_{{\cal O}_0}}(0,0,0,0,0)$
has the characteristic polynomial $(t^2+\beta^2)^2$, which has non-distinct eigenvalues. Its stability property can be established using an algebraic method (see \cite{aeyels}, \cite{dan1}, \cite{dan2}, \cite{dan3}). More precisely, the system of algebraic equations
$$H(x_1,y_1,x_2,y_2,z)=H(0,0,0,0,0),I(x_1,y_1,x_2,y_2,z)=I(0,0,0,0,0),C(x_1,y_1,x_2,y_2,z)=C(0,0,0,0,0)$$
has as unique solution the equilibrium $(0,0,0,0,0)$, leading to the following stability result.

\begin{thm}
The equilibrium $(0,0,0,0,0)$ is degenerate and stable with respect to the dynamics \eqref{sistem5}.
\end{thm}

\section{Homoclinic orbits}

In this section we will give an explicit form of the homoclinic orbits for the unstable equilibria of focus-focus type. This type of equilibria belong to symplectic orbits ${\cal O}_c$ with $c>0$. 

In order to compute the homoclinic orbits, we introduce a local system of coordinates around the equilibrium point $e_c=(0,0,0,0,c)\in {\cal O}_c$. The local system of coordinates is given by
$$\Phi:\R^5\rightarrow \R^5,~~~ (r_1,\theta,y_1,y_2,c)\mapsto \left\{\begin{array}{l}
x_1=r_1\cos\theta\\
x_2=r_1\sin\theta\\
y_1=y_1\\
y_2=y_2\\
z=c-\frac{1}{2}r_1^2.\end{array}\right.$$
Freezing the parameter $c$ we obtain the local system of coordinates on the symplectic orbit ${\cal O}_c$ around the equilibrium point $e_c$:
$$\Phi_c:\R^4\rightarrow {\cal O}_c\setminus \{e_c\},~~~ (r_1,\theta,y_1,y_2)\mapsto \left\{\begin{array}{l}
x_1=r_1\cos\theta\\
x_2=r_1\sin\theta\\
y_1=y_1\\
y_2=y_2.\end{array}\right.$$
As we have excluded the equilibrium point $e_c$ we can work under the assumption that $r_1\neq 0$. The advantage of using polar coordinates in the study of bi-focal homoclinic orbits in four dimensions can be ascertained in \cite{fowler}. By a straightforward computation we obtain that the reduced system on the symplectic leaf ${\cal O}_c$ is given by
\begin{equation}
\label{sistem-coord-noi}
\left\{\begin{array}{l}
\dot r_1=y_1\cos\theta+y_2\sin\theta\\
\dot\theta=\displaystyle\frac{y_2\cos\theta-y_1\sin\theta}{r_1}\\
\dot y_1=r_1\cos\theta\left(c-\frac{1}{2}r_1^2\right)\\
\dot y_2=r_1\sin\theta\left(c-\frac{1}{2}r_1^2\right).\end{array}\right.
\end{equation}
Using a continuity argument and the fact that $I(x_1,y_1,x_2,y_2,z)=x_2y_1-x_1y_2$ is a constant of motion we obtain that if there exists a homoclinic it should belong to the connected component of level set $I^{-1}(I(e_c))=I^{-1}(0)$ that contains $e_c$. If a curve $c(t)=(r_1(t),\theta(t),y_1(t),y_2(t))$ is a homoclinic, then it has to be a solution for the system \eqref{sistem-coord-noi} and to satisfy the following equation for all $t$:
$$r_1(t)(y_1(t)\sin\theta(t)-y_2(t)\cos\theta(t))=0.$$
This implies that $\dot\theta(t)=0$ and thus $\theta(t)=\theta_0$  constant for all $t$. By differentiation and substitution we obtain the following second order equation
$$\ddot r_1=r_1\left(c-\frac{1}{2}r_1^2\right).$$
Making the change of variable $r_1=2\sqrt{c}~\tilde{r_1}$ and the time re-parametrization $\sqrt{c}~t=\tilde{t}$ we obtain the equation
$$\ddot{\tilde{r}}_1(\tilde{t})=\tilde{r}_1(\tilde{t})-2\tilde{r}_1^3(\tilde{t}).$$
It is well known that this second order differential equation has as solutions $\pm\hbox{cn} (\tilde{t},1)=\pm\hbox{sech} (\tilde{t})$. Consequently, we obtain $r_1(t)=\pm 2\sqrt{c}~\hbox{sech}(\sqrt{c}t).$

Substituting $r_1(t)$ in the expression of the local parametrization $\Phi_c$, and for $z$ in the expression of local parametrization $\Phi$ and integrating for $y_1$ and $y_2$ in \eqref{sistem-coord-noi}
we obtain the homoclinic solutions 
\begin{equation*}
\left\{ \begin{array}{l}
x_1(t)=\pm 2\sqrt{c}~\hbox{sech}(\sqrt{c}t)\cos\theta_0\\
x_2(t)=\pm 2\sqrt{c}~\hbox{sech}(\sqrt{c}t)\sin\theta_0\\
y_1(t)=\mp 2c~\hbox{sech}(\sqrt{c}t)\hbox{tanh}(\sqrt{c}t)\cos\theta_0\\
y_2(t)=\mp 2c~\hbox{sech}(\sqrt{c}t)\hbox{tanh}(\sqrt{c}t)\sin\theta_0\\
z(t)=c(1-2\hbox{sech}^2(\sqrt{c}t)).\end{array}\right.
\end{equation*}
The above homoclinic orbits, using different parametrization and arguments, have been discussed in \cite{holm}, \cite{huang}.

\section{Invariant sets and periodic orbits}

We will look for invariant sets of the system \eqref{sistem5} using the technique presented in \cite{chaos}. We have the following vectorial conserved quantity
${\bf F}:\R^5\to\R^3,{\bf F}(p)=(H(p),I(p),C(p))$. In \cite{chaos}, Theorem 2.3, it has been proved that the set $M_{(2)}^{{\bf F}}=\{p\in\R^5|~\hbox{rank}~\nabla {\bf F}(p)=2\}$ is invariant under the dynamics of the system. By direct computation we obtain that $M_{(2)}^{{\bf F}}=M_1\cup M_2$, where
\begin{align*}
M_1&:=\left\{\left(x_1,y_1,x_2,-\frac{x_1y_1}{x_2},-\frac{y_1^2}{x_2^2}\right)|~x_2\ne 0\right\};\\
M_2&:=\left\{\left(x_1,0,0,y_2,-\frac{y_2^2}{x_1^2}\right)|~x_1\ne 0\right\}.
\end{align*}
The union $M_1\cup M_2$, which is a connected set in $\R^5$, is invariant under the dynamics \eqref{sistem5}, but neither the set $M_1$ nor the set $M_2$ are invariant under this dynamics. The vector field corresponding to \eqref{sistem5} is tangent to the sub-manifold $M_1$ and the restricted dynamics on $M_1$ is given by
\begin{equation}\label{sistemmic}
\left\{\begin{array}{l}
\dot x_1=y_1\\
\dot y_1=-\frac{x_1y_1^2}{x_2^2}\\
\dot x_2=-\frac{x_1y_1}{x_2}.
\end{array}\right.
\end{equation}
We notice that the above dynamical system has two conserved quantities, $f_1,f_2:M_1\to\R$, $f_1(x_1,y_1,x_2)=x_1^2+x_2^2$ and $f_2(x_1,y_1,x_2)=\frac{y_1}{x_2}$. Using these conserved quantities and choosing an initial condition $x_1^0,y_1^0,x_2^0$ with $x_2^0\ne 0$ and $y_1^0\ne 0$ we can explicitly solve the system \eqref{sistemmic}:
\begin{equation*}
\left\{\begin{array}{l}
x_1(t)=x_2^0\sin\left(\frac{y_1^0}{x_2^0}t\right)+x_1^0\cos\left(\frac{y_1^0}{x_2^0}t\right)\\
y_1(t)=-\frac{y_1^0}{x_2^0}\left(x_1^0\sin\left(\frac{y_1^0}{x_2^0}t\right)-x_2^0\cos\left(\frac{y_1^0}{x_2^0}t\right)\right)\\
x_2(t)=-x_1^0\sin\left(\frac{y_1^0t}{x_2^0}t\right)+x_2^0\cos\left(\frac{y_1^0}{x_2^0}t\right).
\end{array}\right.
\end{equation*}
Notice that if $y_1^0 =0$ we obtain as constant solutions the equilibrium points from ${\cal E}_2\subset M_1$. The above solution is defined on time intervals $(t_k,t_{k+1})$, where $t_k=\frac{x_2^0}{y_1^0}\vartheta +k\pi \frac{x_2^0}{y_1^0}$ with $k\in\mathbb{Z}$ and $\vartheta\in [0,2\pi)$ is the unique real number such that $\sin\vartheta =\frac{x_2^0}{\sqrt{(x_1^0)^2+(x_2^0)^2}}$ and $ \cos\vartheta =\frac{x_1^0}{\sqrt{(x_1^0)^2+(x_2^0)^2}}$.
For time values $t_k$ the above solution exits the set $M_1$ and punctures the set $M_2$, thus making the union $M_1\cup M_2$ an invariant set. Although the solution starting from $M_1$ is not complete, we can construct a complete periodic solution for the initial system \eqref{sistem5} given by  
\begin{equation*}
\left\{\begin{array}{l}
x_1(t)=x_2^0\sin\left(\frac{y_1^0}{x_2^0}t\right)+x_1^0\cos\left(\frac{y_1^0}{x_2^0}t\right)\\
y_1(t)=-\frac{y_1^0}{x_2^0}\left(x_1^0\sin\left(\frac{y_1^0}{x_2^0}t\right)-x_2^0\cos\left(\frac{y_1^0}{x_2^0}t\right)\right)\\
x_2(t)=-x_1^0\sin\left(\frac{y_1^0}{x_2^0}t\right)+x_2^0\cos\left(\frac{y_1^0}{x_2^0}t\right)\\
y_2(t)=-\frac{y_1^0}{x_2^0}\left(x_2^0\sin\left(\frac{y_1^0}{x_2^0}t\right)+x_1^0\cos\left(\frac{y_1^0}{x_2^0}t\right)\right)\\
z(t)=-\frac{(y_1^0)^2}{(x_2^0)^2}.
\end{array}\right.
\end{equation*}  

\medskip

\noindent {\bf Acknowledgements.} Petre Birtea was supported by a grant of the Romanian National Authority for Scientific Research, CNCS – UEFISCDI, project number PN-II-RU-TE-2011-3-0006. Ioan Ca\c su was supported by a grant of the Romanian National Authority for Scientific Research, CNCS – UEFISCDI, project number PN-II-ID-PCE-2011-3-0571.

\end{document}